\documentclass[12pt, one side,emlines]{amsart}
\usepackage[a4paper,margin=1in]{geometry}
\usepackage{fancyhdr}
\usepackage{enumerate}
\usepackage{amssymb,latexsym,xy,eucal,mathrsfs}
\textwidth=18cm \textheight=27cm \theoremstyle{plain}
\usepackage{tikz}
\usetikzlibrary{positioning}
\theoremstyle{definition}

\newtheorem{theorem}{Theorem}[section]
\newtheorem{thm}[theorem]{Theorem}
\newtheorem{lem}[theorem]{Lemma}

\newtheorem{defn}{Definition}[section]

\begin{document}
	

	%
	\title{Graphic Elementary Lift of Cographic Matroids}\maketitle
	
	\markboth{ Shital Dilip Solanki, Ganesh Mundhe and S. B. Dhotre}{Graphic Elementary Lift of Cographic Matroids }\begin{center}\begin{large} Shital Dilip Solanki$^1$, Ganesh Mundhe$^2$ and S. B. Dhotre$^3$ \end{large}\\\begin{small}\vskip.1in\emph{
				1. Ajeenkya DY Patil University, Pune-411047, Maharashtra,
				India.\\ 
				2. Army Institute of Technology, Pune-411015, Maharashtra,
				India.\\
				3. Department of Mathematics,
				Savitribai Phule Pune University,\\ Pune - 411007, Maharashtra,
				India.}\\
			E-mail: \texttt{1. shital.solanki@adypu.edu.in, 2. gmundhe@aitpune.edu.in, 3. dsantosh2@yahoo.co.in. }\end{small}\end{center}\vskip.2in

	\begin{abstract} 
	 A matroid $N$ is a lift of a binary matroid $M$, if $N=Q\backslash X$ when $Q/X=M$ for some binary matroid $Q$ and $X \subseteq E(Q)$ and is called an elementary lift of $M$, if $|X|=1$. A splitting operation on a binary matroid can result in an elementary lift. An elementary lift of a cographic or a graphic matroid need not be cographic or graphic. We intend to characterize the cographic matroids whose elementary lift is a graphic matroid.   
	\end{abstract}\vskip.2in
	
	\noindent\begin{Small}\textbf{Mathematics Subject Classification (2010)}:
		05C83, 05C50, 05B35     \end{Small}\vskip.2in
	\vskip.25in

	\baselineskip 19truept 
	\section{Introduction}
	\noindent Oxley \cite{ox} to be referred for vague concepts and notations. 
	For a matroid $M$, if there is matroid $N$ such that $N=Q\backslash X$ if $Q/X=M$ for some binary matroid $Q$ and $X\subseteq E(Q)$, then $N$ is called lift of $M$ and is called an elementary lift if $|X|=1$. The splitting operation results in an elementary lift. The splitting operation in the graph was introduced by Fleischner \cite{fl}. Later, Raghunathan et al. \cite{ttr} defines splitting for binary matroids. Splitting is then generalized using a set by Shikare et al. \cite{mms1} as follows.
	\begin{defn} 
		Let a binary matroid $M$ represented by a matrix $A$. Append a row at the bottom of $A$ with entries 1 corresponding to the elements of $S$ and 0 everywhere else, where $S\subseteq E(M)$. Let the matrix be $A_S$. Then $M_S= M(A_S)$ is the splitting matroid, and the operation is called the splitting operation using set $S$.  
	\end{defn}
  The matroid $B_S$ need not be cographic or graphic for a cographic binary matroid $B$. Thus, the splitting operation does not protect matroid properties like graphicness, cographicness, etc. N. Pirouz \cite{np_ctg} characterized a cographic matroid whose splitting using two elements is graphic. In the following theorem, Ganesh et al. \cite{gm} characterized graphic matroid whose splitting matroid, using three elements, is graphic.
\begin{thm}\cite{gm}\label{gtg3thm}
	Let $S \subseteq E(M)$, with $|S|=3$, where $M$ is a graphic binary matroid, then $M_S$ is graphic if and only if the matroid $M$ do not have minors $M(F_i)$, where the Figure \ref{g2g}, shows the graph $F_i$, for $i=1,2 \cdots 7$. 
\end{thm}  
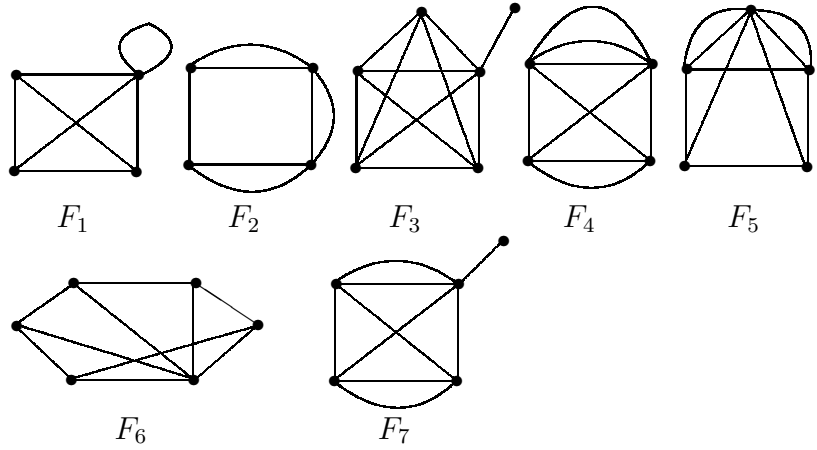
\begin{figure}[h!]
	\centering
\unitlength 1mm 
\linethickness{0.4pt}
\ifx\plotpoint\undefined\newsavebox{\plotpoint}\fi 
\begin{picture}(107.35,64.642)(0,0)
	\put(24.924,36.199){\circle*{1.514}}
	\put(41.006,36.17){\circle*{1.514}}
	\put(25.216,49.028){\circle*{1.514}}
	\put(41.298,48.999){\circle*{1.514}}
	\put(25.088,49.033){\line(1,0){16.2}}
	\put(25.021,49.029){\line(0,-1){12.713}}
	\put(41.109,36.092){\line(0,1){13.162}}
	\qbezier(24.924,49.253)(33.343,55.119)(41.385,48.875)
	\put(25.113,36.199){\line(1,0){16.271}}
	\qbezier(24.735,36.199)(33.628,29.009)(41.006,36.199)
	\put(90.189,35.99){\circle*{1.514}}
	\put(106.271,35.96){\circle*{1.514}}
	\put(90.48,48.819){\circle*{1.514}}
	\put(106.562,48.79){\circle*{1.514}}
	\put(90.353,48.824){\line(1,0){16.2}}
	\put(90.286,48.821){\line(0,-1){12.712}}
	\put(106.373,35.883){\line(0,1){13.163}}
	\put(90.378,35.99){\line(1,0){16.271}}
	\put(98.885,56.715){\circle*{1.514}}
	\multiput(98.931,56.895)(.033596413,-.036){223}{\line(0,-1){.036}}
	\multiput(90.268,48.919)(.036548523,.033700422){237}{\line(1,0){.036548523}}
	\multiput(90.043,36.094)(.0337078652,.0779475655){267}{\line(0,1){.0779475655}}
	\put(1.974,35.315){\circle*{1.514}}
	\put(18.056,35.286){\circle*{1.514}}
	\put(2.265,48.144){\circle*{1.514}}
	\put(18.348,48.115){\circle*{1.514}}
	\put(2.138,48.149){\line(1,0){16.2}}
	\put(2.071,48.146){\line(0,-1){12.712}}
	\put(18.159,35.208){\line(0,1){13.163}}
	\put(2.163,35.315){\line(1,0){16.271}}
	\multiput(2.068,48.181)(.0418534031,-.0336806283){382}{\line(1,0){.0418534031}}
	\multiput(1.785,35.315)(.0433376963,.0336806283){382}{\line(1,0){.0433376963}}
	\multiput(98.575,56.925)(.033597285,-.094683258){221}{\line(0,-1){.094683258}}
	\qbezier(18.092,48.375)(12.805,52.763)(19.667,54.9)
	\qbezier(19.667,54.9)(26.192,51.638)(18.317,47.925)
	\qbezier(41.267,48.825)(46.779,42.525)(41.042,36.225)
	\qbezier(90.025,48.825)(90.925,57.937)(99.025,56.7)
	\qbezier(99.025,56.7)(107.35,56.587)(106.675,48.825)
	\put(32.2,29){\makebox(0,0)[cc]{$F_2$}}
	\put(53.35,29){\makebox(0,0)[cc]{$F_3$}}
	\put(76.3,29){\makebox(0,0)[cc]{$F_4$}}
	\put(97.9,29){\makebox(0,0)[cc]{$F_5$}}
	\put(9.507,7.677){\circle*{1.514}}
	\put(25.589,7.648){\circle*{1.514}}
	\put(9.798,20.506){\circle*{1.514}}
	\put(25.881,20.477){\circle*{1.514}}
	\put(9.671,20.511){\line(1,0){16.2}}
	\put(9.696,7.677){\line(1,0){16.271}}
	\multiput(9.601,20.543)(.0418534031,-.0336806283){382}{\line(1,0){.0418534031}}
	\put(2.307,14.857){\circle*{1.514}}
	\put(34.032,14.969){\circle*{1.514}}
	\multiput(25.738,7.687)(.038364055,.033700461){217}{\line(1,0){.038364055}}
	\put(34.063,15){\line(-3,2){8.438}}
	\multiput(2.225,15.112)(.048006369,.033681529){157}{\line(1,0){.048006369}}
	\multiput(2.112,14.775)(.035328502,-.033695652){207}{\line(1,0){.035328502}}
	\multiput(2.037,15.025)(.109345794,-.03364486){214}{\line(1,0){.109345794}}
	\put(25.437,7.825){\line(0,1){13.05}}
	\multiput(9.462,7.825)(.117391304,.033695652){207}{\line(1,0){.117391304}}
	\put(17.337,1){\makebox(0,0)[cc]{$F_6$}}
	\put(44.067,7.544){\circle*{1.514}}
	\put(60.149,7.515){\circle*{1.514}}
	\put(44.358,20.373){\circle*{1.514}}
	\put(60.441,20.344){\circle*{1.514}}
	\put(44.231,20.378){\line(1,0){16.2}}
	\put(44.164,20.375){\line(0,-1){12.712}}
	\put(60.252,7.437){\line(0,1){13.163}}
	\qbezier(44.067,20.599)(52.486,26.464)(60.528,20.221)
	\put(44.256,7.544){\line(1,0){16.271}}
	\qbezier(43.878,7.544)(52.771,.354)(60.149,7.544)
	\multiput(44.161,20.41)(.0418534031,-.0336806283){382}{\line(1,0){.0418534031}}
	\multiput(43.878,7.544)(.0433376963,.0336806283){382}{\line(1,0){.0433376963}}
	\put(9.723,29){\makebox(0,0)[cc]{$F_1$}}
	\put(51.972,1){\makebox(0,0)[cc]{$F_7$}}
	\put(66.341,26.036){\circle*{1.514}}
	\multiput(60.458,20.329)(.033716763,.033722543){173}{\line(0,1){.033722543}}
	\put(69.54,36.711){\circle*{1.514}}
	\put(85.622,36.682){\circle*{1.514}}
	\put(69.831,49.54){\circle*{1.514}}
	\put(85.914,49.511){\circle*{1.514}}
	\put(69.704,49.545){\line(1,0){16.2}}
	\put(69.637,49.542){\line(0,-1){12.712}}
	\put(85.725,36.604){\line(0,1){13.163}}
	\qbezier(69.54,49.766)(77.959,55.631)(86.001,49.388)
	\put(69.729,36.711){\line(1,0){16.271}}
	\qbezier(69.351,36.711)(78.244,29.521)(85.622,36.711)
	\multiput(69.634,49.577)(.0418534031,-.0336806283){382}{\line(1,0){.0418534031}}
	\multiput(69.351,36.711)(.0433376963,.0336806283){382}{\line(1,0){.0433376963}}
	\put(46.895,35.744){\circle*{1.514}}
	\put(62.977,35.714){\circle*{1.514}}
	\put(47.186,48.573){\circle*{1.514}}
	\put(63.269,48.544){\circle*{1.514}}
	\put(47.059,48.578){\line(1,0){16.2}}
	\put(46.992,48.575){\line(0,-1){12.712}}
	\put(63.08,35.637){\line(0,1){13.163}}
	\put(47.084,35.744){\line(1,0){16.271}}
	\put(55.592,56.469){\circle*{1.514}}
	\multiput(55.637,56.649)(.033596413,-.036){223}{\line(0,-1){.036}}
	\multiput(46.974,48.673)(.036552743,.033700422){237}{\line(1,0){.036552743}}
	\multiput(46.749,35.848)(.0337078652,.0779475655){267}{\line(0,1){.0779475655}}
	\multiput(55.363,56.454)(.033597285,-.093665158){221}{\line(0,-1){.093665158}}
	\multiput(62.788,35.754)(-.0419291339,.0336614173){381}{\line(-1,0){.0419291339}}
	\multiput(46.813,35.979)(.0439171123,.0336898396){374}{\line(1,0){.0439171123}}
	\put(67.858,57.002){\circle*{1.514}}
	\multiput(63.212,48.644)(.033547445,.061934307){137}{\line(0,1){.061934307}}
	\qbezier(69.4,49.704)(78.504,64.642)(85.84,49.527)
\end{picture}

	\caption{Excluded minors for the splitting of a graphic matroid using three elements.}
	\label{g2g}
\end{figure}

\noindent Let $\mathcal{C}_k$ be the collection of cographic matroid whose splitting using $k$ elements is graphic. It is observed that there is no minimal minor  $E$ such that $E \notin \mathcal{C}_1$.

\noindent N. Pirouz \cite{np_ctg} characterized the class $\mathcal{C}_2$.
\begin{thm}\cite{np_ctg}\label{ctg2el}
	Let $C$ be a cographic binary matroid, then $C \in \mathcal{C}_2$ if and only if it does not have $M(G_1)$ or $M(G_2)$ minor, Figure \ref{Fig_gtg_2elt} shows the graphs $G_1$ and $G_2$.
\end{thm}
\begin{figure}[h!]
	\centering
	\unitlength 1mm 
	\linethickness{0.4pt}
	\ifx\plotpoint\undefined\newsavebox{\plotpoint}\fi 
	\begin{picture}(47.194,31.693)(0,0)
		\put(3.569,8.056){\circle*{1.682}}
		\put(21.438,8.024){\circle*{1.682}}
		\put(3.893,22.31){\circle*{1.682}}
		\put(21.762,22.278){\circle*{1.682}}
		\put(3.751,22.316){\line(1,0){18}}
		\put(3.677,22.312){\line(0,-1){14.125}}
		\put(21.552,7.937){\line(0,1){14.625}}
		\qbezier(3.569,22.561)(12.924,29.078)(21.859,22.141)
		\put(3.779,8.056){\line(1,0){18.079}}
		\qbezier(3.359,8.056)(13.24,.067)(21.438,8.056)
		\multiput(3.674,22.351)(.0418962264,-.0337146226){424}{\line(1,0){.0418962264}}
		\multiput(3.359,8.056)(.043384434,.0337146226){424}{\line(1,0){.043384434}}
		\put(28.16,7.824){\circle*{1.682}}
		\put(46.029,7.791){\circle*{1.682}}
		\put(28.484,22.078){\circle*{1.682}}
		\put(46.353,22.046){\circle*{1.682}}
		\put(28.342,22.084){\line(1,0){18}}
		\put(28.268,22.08){\line(0,-1){14.125}}
		\put(46.143,7.705){\line(0,1){14.625}}
		\put(28.37,7.824){\line(1,0){18.079}}
		\qbezier(27.95,7.824)(37.83,-.165)(46.029,7.824)
		\put(37.823,30.852){\circle*{1.682}}
		\multiput(37.874,31.052)(.0337004049,-.0361133603){247}{\line(0,-1){.0361133603}}
		\multiput(28.248,22.189)(.0364583333,.0336174242){264}{\line(1,0){.0364583333}}
		\multiput(27.998,7.939)(.0336700337,.0778619529){297}{\line(0,1){.0778619529}}
		\put(13.258,0){\makebox(0,0)[cc]{$G_1$}}
		\put(37.3,0.0){\makebox(0,0)[cc]{$G_2$}}
	\end{picture}
	\caption{Minimal minors not in the class $\mathcal{C}_2$}.
	\label{Fig_gtg_2elt}
\end{figure}
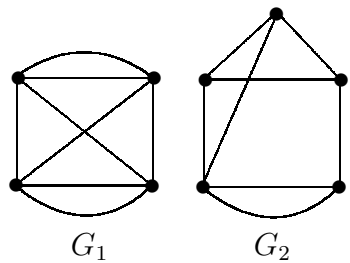

\noindent This paper proves the following theorems. 
\begin{thm}\label{mt1}
A cographic binary matroid $M \notin \mathcal{C}_k$, $k \geq 2$, then $M$ contains a minor $P$ such that one of the below is true.\\
i) $P$ is an extension of a minimal minor $E$ such that $E \notin \mathcal{C}_{k-1}$ by single element. \\
ii)  $P= M(Q_i)$.\\
iii) $P$ is a coextension of $M(Q_i)$ by $n$ elements, where $n \leq k$, the Figure \ref{quotient} shows the graph $Q_i$, for $i=1,2, \cdots 9$.    

\end{thm} 

\noindent We show that the forbidden minors obtained by Mundhe et al. \cite{gm} are the only minimal minors not in the class $\mathcal{C}_3$.
\begin{thm}\label{mt2}
	Let a cographic binary matroid be $M$, then $M \in \mathcal{C}_3$ if and only if $M$ does not have a minor $M(F_i)$, Figure \ref{g2g} shows the graph $F_i$ for $i=1,2, \cdots 7$.
\end{thm}

\section{Preliminary Results}
\noindent We denote $\mathcal{F}=\{F_7^*, M^*(K_{3,3}), F_7, M^*(K_5)\}$. An elementary quotient of $F \in \mathcal{F}$ is denoted by $Q_F$. 
\begin{thm}\cite{ox}\label{gt}
	A binary matroid is a graphic matroid if and only if it does not has a minor $F \in \mathcal{F}$.
\end{thm}
\begin{thm}\cite{ox}\label{cgt}
	A binary matroid is a cographic matroid if and only if it does not has a minor from the set $\{F_7, M(K_5), F_7^*, M(K_{3,3})\}$.
\end{thm}
\noindent In this paper, we use the technique discovered by Mundhe et al. \cite{gm} to find the excluded minors.
The following lemmas are used to prove the main theorems. 
\begin{lem}\label{mainlemma}
	Let $M_S$ is not a graphic binary matroid for a cographic binary matroid $M$ for $S\subseteq E(M)$ and $|S|=k, k\geq 2$. Then there exists a minor $P$ of $M$ such as one of the below is true.\\
	(i) $P_S \cong F$ or $P_S/S' \cong F$, for some $S'\subseteq S$ and $F \in \mathcal{F}$.\\
	(ii) $P$ is an extension of a minimal minor $E$ by an element, where $E \notin \mathcal{C}_{k-1}$.
	\begin{proof}
		On a similar line of the proof of Lemma 3.3 in \cite{gm}.
	\end{proof} 
\end{lem}

\begin{lem}
	Let $P$ be the minor as stated in Lemma \ref{mainlemma}(i). Then $P$ does not contain a coloop. 
\end{lem}
\begin{proof}
	On a similar line of the proof of Lemma 3.4 in \cite{gm}.
\end{proof}
\begin{lem}\label{rel}
	Let $P$ be the minor as stated in Lemma \ref{mainlemma}(i), without containing a coloop. Then $P$ is a coextension of $Q_F$ by $n$ elements, where $n \leq k$, $k \geq 2$ or $P\cong Q_F$ for some binary matroid $N$ with $a \in E(N)$, such that $N\backslash a\cong F$ for $F \in \mathcal{F}$. 
\end{lem}
\begin{proof}
	On a similar line of the proof of Lemma 3.6 in \cite{gm}.
\end{proof}
\noindent From the definition of an elementary quotient and above lemma, $N/a=Q_F$, $F \in \mathcal{F}$. Thus, we need quotients of every $F \in \mathcal{F}$ to find excluded minors for the class $\mathcal{C}_k$, for $k \geq 2$,  Mundhe et al. \cite{gm} obtained graphic quotients for every $F \in \mathcal{F}$ as follows.
\begin{lem}\cite{gm}\label{qF7*}
	A graphic elementary quotient $Q_{F_7^*} \cong M(Q_1)$ or $Q_{F_7^*} \cong M(Q_2)$. The graphs $Q_1$ and $Q_2$ are shown in Figure \ref{quotient}.
\end{lem}
\begin{lem}\cite{gm}\label{qF7}
	A graphic elementary quotient $Q_{F_7} \cong M(Q_3)$. The graph $Q_3$ is shown in Figure \ref{quotient}.
\end{lem}
\begin{lem}\cite{gm}\label{gqm*k33}
	A graphic elementary quotient $Q_{M^*(K_{3,3})} \cong M(Q_4)$ or $Q_{M^*(K_{3,3})} \cong M(Q_5)$. The graphs $Q_4$ and $Q_5$ are shown in Figure \ref{quotient}.
\end{lem}
\begin{lem}\cite{gm}\label{gqm*k5}
	A graphic elementary quotient $Q_{M^*(K_5)} \cong M(Q_i)$, the graph $Q_i$ is as given in Figure \ref{quotient}, for $i=6,7,8,9$.
\end{lem}
\begin{figure}[h!]
	\centering
	\unitlength 1mm 
	\linethickness{0.4pt}
	\ifx\plotpoint\undefined\newsavebox{\plotpoint}\fi 
	\begin{picture}(142.344,61.49)(0,0)
		\put(3.359,36.436){\circle*{1.682}}
		\put(21.228,36.404){\circle*{1.682}}
		\put(3.683,50.69){\circle*{1.682}}
		\put(21.552,50.658){\circle*{1.682}}
		\put(3.541,50.696){\line(1,0){18}}
		\put(3.467,50.692){\line(0,-1){14.125}}
		\put(21.342,36.317){\line(0,1){14.625}}
		\put(3.569,36.436){\line(1,0){18.079}}
		\multiput(3.464,50.731)(.0418962264,-.0337146226){424}{\line(1,0){.0418962264}}
		\multiput(3.149,36.436)(.043384434,.0337146226){424}{\line(1,0){.043384434}}
		\put(32.9,36.381){\circle*{1.682}}
		\put(50.769,36.348){\circle*{1.682}}
		\put(33.224,50.635){\circle*{1.682}}
		\put(51.093,50.603){\circle*{1.682}}
		\put(33.082,50.641){\line(1,0){18}}
		\put(33.008,50.637){\line(0,-1){14.125}}
		\put(50.883,36.262){\line(0,1){14.625}}
		\put(33.11,36.381){\line(1,0){18.079}}
		\qbezier(32.69,36.381)(42.57,28.392)(50.769,36.381)
		\qbezier(33.201,50.654)(25.599,45.35)(32.847,36.512)
		\qbezier(33.201,50.83)(42.393,59.051)(50.878,50.654)
		\qbezier(21.357,50.83)(15.169,58.432)(24.539,56.841)
		\qbezier(24.539,56.841)(29.488,52.687)(21.71,51.007)
		\put(58.862,37.241){\circle*{1.682}}
		\put(76.731,37.209){\circle*{1.682}}
		\put(59.072,37.241){\line(1,0){18.079}}
		\put(66.85,52.135){\circle*{1.682}}
		\multiput(58.652,36.999)(.033736626,.063152263){243}{\line(0,1){.063152263}}
		\multiput(66.85,52.345)(.0337235495,-.050225256){293}{\line(0,-1){.050225256}}
		\qbezier(58.652,37.419)(55.078,50.243)(66.64,52.135)
		\qbezier(76.731,37.419)(64.433,27.434)(58.441,37.209)
		\qbezier(66.64,52.345)(79.358,49.612)(76.52,37.629)
		\qbezier(66.64,52.135)(58.231,56.76)(66.64,58.862)
		\qbezier(66.64,58.862)(72.842,57.075)(66.85,52.345)
		\put(86.401,36.611){\circle*{1.682}}
		\put(104.27,36.579){\circle*{1.682}}
		\put(86.725,50.865){\circle*{1.682}}
		\put(104.594,50.833){\circle*{1.682}}
		\put(86.583,50.871){\line(1,0){18}}
		\put(86.509,50.867){\line(0,-1){14.125}}
		\put(104.384,36.492){\line(0,1){14.625}}
		\put(86.611,36.611){\line(1,0){18.079}}
		\multiput(86.506,50.906)(.0418962264,-.0337146226){424}{\line(1,0){.0418962264}}
		\multiput(86.191,36.611)(.043384434,.0337146226){424}{\line(1,0){.043384434}}
		\put(114.36,36.611){\circle*{1.682}}
		\put(132.229,36.579){\circle*{1.682}}
		\put(114.684,50.865){\circle*{1.682}}
		\put(132.553,50.833){\circle*{1.682}}
		\put(114.542,50.871){\line(1,0){18}}
		\put(114.468,50.867){\line(0,-1){14.125}}
		\put(132.343,36.492){\line(0,1){14.625}}
		\put(114.57,36.611){\line(1,0){18.079}}
		\multiput(114.465,50.906)(.0418962264,-.0337146226){424}{\line(1,0){.0418962264}}
		\multiput(114.15,36.611)(.043384434,.0337146226){424}{\line(1,0){.043384434}}
		\qbezier(114.36,50.663)(123.295,61.49)(132.649,50.873)
		\qbezier(114.36,36.578)(124.451,26.803)(132.019,36.789)
		\qbezier(86.611,50.873)(95.861,60.649)(104.27,50.663)
		\qbezier(86.611,50.873)(78.412,45.303)(86.191,36.789)
		\qbezier(86.401,50.873)(92.707,35.212)(104.9,36.789)
		\put(8.588,6.042){\circle*{1.514}}
		\put(24.67,6.013){\circle*{1.514}}
		\put(8.879,18.871){\circle*{1.514}}
		\put(24.962,18.842){\circle*{1.514}}
		\put(8.752,18.876){\line(1,0){16.2}}
		\put(8.777,6.042){\line(1,0){16.271}}
		\multiput(8.682,18.908)(.0418534031,-.0336806283){382}{\line(1,0){.0418534031}}
		\put(1.388,13.222){\circle*{1.514}}
		\put(33.113,13.334){\circle*{1.514}}
		\multiput(24.819,6.052)(.038364055,.033700461){217}{\line(1,0){.038364055}}
		\put(33.144,13.365){\line(-3,2){8.438}}
		\multiput(1.306,13.477)(.048006369,.033681529){157}{\line(1,0){.048006369}}
		\multiput(1.193,13.14)(.035328502,-.033695652){207}{\line(1,0){.035328502}}
		\multiput(1.118,13.39)(.109345794,-.03364486){214}{\line(1,0){.109345794}}
		\put(24.518,6.19){\line(0,1){13.05}}
		\multiput(8.543,6.19)(.117391304,.033695652){207}{\line(1,0){.117391304}}
		\put(45.166,6.042){\circle*{1.514}}
		\put(61.248,6.013){\circle*{1.514}}
		\put(45.457,18.871){\circle*{1.514}}
		\put(61.54,18.842){\circle*{1.514}}
		\put(45.33,18.876){\line(1,0){16.2}}
		\put(45.355,6.042){\line(1,0){16.271}}
		\put(37.966,13.222){\circle*{1.514}}
		\put(69.691,13.334){\circle*{1.514}}
		\multiput(61.397,6.052)(.038364055,.033700461){217}{\line(1,0){.038364055}}
		\put(69.722,13.365){\line(-3,2){8.438}}
		\multiput(37.884,13.477)(.048006369,.033681529){157}{\line(1,0){.048006369}}
		\multiput(37.771,13.14)(.035328502,-.033695652){207}{\line(1,0){.035328502}}
		\put(61.096,6.19){\line(0,1){13.05}}
		\put(81.114,6.042){\circle*{1.514}}
		\put(97.196,6.013){\circle*{1.514}}
		\put(81.405,18.871){\circle*{1.514}}
		\put(97.488,18.842){\circle*{1.514}}
		\put(81.278,18.876){\line(1,0){16.2}}
		\put(81.303,6.042){\line(1,0){16.271}}
		\put(73.914,13.222){\circle*{1.514}}
		\put(105.639,13.334){\circle*{1.514}}
		\multiput(97.345,6.052)(.038364055,.033700461){217}{\line(1,0){.038364055}}
		\put(105.67,13.365){\line(-3,2){8.438}}
		\multiput(73.832,13.477)(.048006369,.033681529){157}{\line(1,0){.048006369}}
		\multiput(73.719,13.14)(.035328502,-.033695652){207}{\line(1,0){.035328502}}
		\put(117.062,6.252){\circle*{1.514}}
		\put(133.144,6.223){\circle*{1.514}}
		\put(117.353,19.081){\circle*{1.514}}
		\put(133.436,19.052){\circle*{1.514}}
		\put(117.226,19.086){\line(1,0){16.2}}
		\put(117.251,6.252){\line(1,0){16.271}}
		\put(109.862,13.432){\circle*{1.514}}
		\put(141.587,13.544){\circle*{1.514}}
		\multiput(133.293,6.262)(.038364055,.033700461){217}{\line(1,0){.038364055}}
		\put(141.618,13.575){\line(-3,2){8.438}}
		\multiput(109.78,13.687)(.048006369,.033681529){157}{\line(1,0){.048006369}}
		\multiput(109.667,13.35)(.035328502,-.033695652){207}{\line(1,0){.035328502}}
		\put(45.408,18.92){\line(0,-1){13.244}}
		\put(37.63,13.244){\line(1,0){31.743}}
		\put(73.788,13.454){\line(1,0){31.954}}
		\multiput(80.935,19.13)(.0411882952,-.0336997455){393}{\line(1,0){.0411882952}}
		\multiput(80.725,5.886)(.0441417323,.0336587927){381}{\line(1,0){.0441417323}}
		\multiput(117.093,19.34)(.0412842377,-.0336770026){387}{\line(1,0){.0412842377}}
		\multiput(117.093,6.096)(.042369509,.0336795866){387}{\line(1,0){.042369509}}
		\qbezier(109.735,13.454)(109.42,20.707)(117.093,19.13)
		\qbezier(133.49,19.13)(140.743,21.127)(141.689,13.454)
		\qbezier(81.145,18.92)(89.764,25.752)(97.543,18.71)
		\qbezier(61.385,19.13)(59.282,26.383)(64.748,24.806)
		\qbezier(64.748,24.806)(69.268,21.337)(61.595,19.13)
		\put(66.85,28){\makebox(0,0)[cc]{$Q_3$}}
		\put(12.613,28){\makebox(0,0)[cc]{$Q_1$}}
		\put(41.624,28){\makebox(0,0)[cc]{$Q_2$}}
		\put(95.44,28){\makebox(0,0)[cc]{$Q_4$}}
		\put(124.03,28){\makebox(0,0)[cc]{$Q_5$}}
		\put(16.397,1){\makebox(0,0)[cc]{$Q_6$}}
		\put(53.396,1){\makebox(0,0)[cc]{$Q_7$}}
		\put(89.554,1){\makebox(0,0)[cc]{$Q_8$}}
		\put(125.081,1){\makebox(0,0)[cc]{$Q_9$}}
		\qbezier(132.439,51.084)(130.863,58.862)(136.433,56.549)
		\qbezier(136.433,56.549)(140.007,51.294)(132.649,51.084)
	\end{picture}
	
	\caption{Graphic quotients of Non-graphic matroids}
	\label{quotient}
\end{figure}
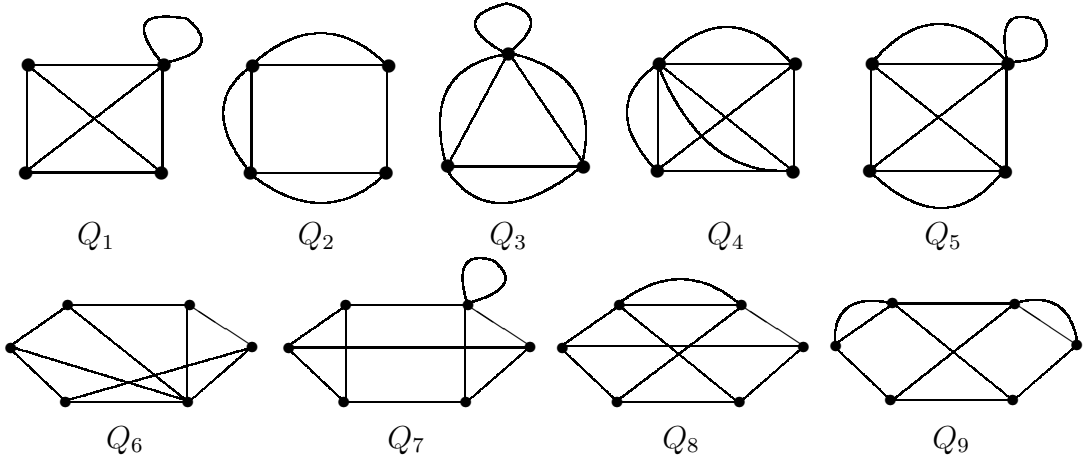
\noindent To find excluded minors for $\mathcal{C}_k$ for $k \geq 3$, we need graphic and non-graphic quotients for every $F \in \mathcal{F}$. Mundhe found graphic quotients for every $F \in \mathcal{F}$. This paper finds non-graphic quotients for every  $F \in \mathcal{F}$. Since either quotient $Q$ or coextension of $Q$ will be the minor of a cographic matroid, by Theorem \ref{cgt}, $Q$ should not contain $F_7$ and $F_7^*$. Thus, for every $F \in \mathcal{F}$, we find non-graphic quotients not containing $F_7$ and $F_7^*$.  
We proved the following lemmas.
\begin{lem}\label{ngqF7}
	A quotient $Q_{F_7}$ not containing $F_7$ and $F_7^*$ is graphic.	
\end{lem}
\begin{proof}
	Let $N\backslash a \cong F_7$, where $N$ is a binary matroid and $a \in E(N)$, then $Q_{F_7}=N/a$, if $a$ is a coloop or a loop, then $Q_{F_7}= N\backslash a \cong F_7$, thus $Q_{F_7} \cong F_7$, a contradiction. If $a$ is not a coloop or a loop, then $r( N\backslash a)=r(F_7)=3$. Thus $r(Q_{F_7})=2$. Hence, $Q_{F_7}$ can not have a minor $F \in \{ M^*(K_{3,3}), M^*(K_5)\}$, as $r(F) \geq 3$. Thus, by Theorem \ref{gt}, $Q_{F_7}$ is graphic. 
\end{proof}
\begin{lem}\label{ngqF7^*}
	A quotient $Q_{F_7^*}$ not containing $F_7$ and $F_7^*$ is graphic.	
\end{lem}
\begin{proof}
	Let $N\backslash a \cong F_7^*$, for a binary matroid $N$ having an element $a$, then $Q_{F_7^*} =N/a$, if $a$ is a coloop or a loop, then $Q_{F_7^*}= N\backslash a \cong F_7^*$, thus $Q_{F_7^*} \cong F_7^*$, a contradiction. If $a$ is not a coloop or a loop, then $r( N\backslash a)=r(F_7^*)=4$. $Q_{F_7^*}$ can not contain $F \in \{ M^*(K_{3,3}), M^*(K_5)\}$ minor, as $r(F) \ge 4$ and $r(Q_{F_7^*})=3$. Thus, by Theorem \ref{gt}, $Q_{F_7^*}$ is graphic.
\end{proof}
\begin{lem}\label{qk5}
	Let a binary matroid be $N$ having an element $a$ such that $a$ is not a loop or coloop, then $Q_{M^*(K_5)}$ not containing $F_7$ or $F_7^*$ is graphic. 
\end{lem}
\begin{proof}
	Let a binary matroid be $N$ having an element $a$ and $a$ is not a coloop or loop, such that $N\backslash a \cong M^*(K_5)$ then $r(N\backslash a)=6$ and $E(N\backslash a)=10$, then $r(N)=6$ and $E(N)=11$ then $Q_{M^*(K_5)}=N/a$. Thus, $r(Q_{M^*(K_5)})=5$ and $E(Q_{M^*(K_5)})=10$.  Suppose $Q_{M^*(K_5)}$ is not graphic. Then  by Theorem \ref{gt}, $Q_{M^*(K_5)}$ contains $M^*(K_{3,3})$ or  $M^*(K_5)$  minor.\\ $Q_{M^*(K_5)}$ does not contains $M^*(K_5)$ minor, as $r(Q_{M^*(K_5)})=5$ and $r(M^*(K_5))=6$.
	If $Q_{M^*(K_5)}$ contains $M^*(K_{3,3})$ minor, then $Q_{M^*(K_5)} \backslash A_1/A_2 \cong M^*(K_{3,3})$ for some subsets $A_1$ and $A_2$ of $E(Q_{M^*(K_5)})$.
	a) If $A_1=\emptyset$ and $A_2=\emptyset$ then $Q_{M^*(K_5)} \cong M^*(K_{3,3})$, a contradiction, as $r(Q_{M^*(K_5)})=5$ and $r(M^*(K_{3,3}))=4$.
	b) If $A_1=\emptyset$ and $A_2 \neq \emptyset$, then, if $|A_2| > 1$, $Q_{M^*(K_5)}/A_2 \cong M^*(K_{3,3})$ a contradiction, as $r(Q_{M^*(K_5)}/A_2)\leq 3$ and $r(M^*(K_{3,3}))=4$. If $|A_2|=1$, then $Q_{M^*(K_5)}/b \cong M^*(K_{3,3})$ that is $N/a/b \cong M^*(K_{3,3})$ for some $b \in E(Q_{M^*(K_5)})$. Thus, $(N/a/b)^* \cong N^*\backslash a \backslash b \cong M(K_{3,3})$. Also, we have $N \backslash a \cong M^*(K_5)$ thus $N^*/a \cong M(K_5)$. $M(K_{3,3})$ contains more than six odd cocircuits. Hence, $N^*$ contains at least two odd cocircuits without containing $a$. Therefore $N^*/a$ contains at least one odd cocircuit, a contradiction as $N^*/a \cong M(K_5)$ and $M(K_5)$ is Eulerian.
	c) If $A_1 \neq \emptyset$ and $A_2=\emptyset$, then $Q_{M^*(K_5)}\backslash A_1 \cong M^*(K_{3,3})$, a contradiction, as $r(Q_{M^*(K_5)}\backslash A_1)=5$ and $r(M^*(K_{3,3}))=4$, when $|A_1|=1$ and $E(Q_{M^*(K_5)}\backslash A_1) \leq 8$ when $|A_1| > 1$, whereas $E(M^*(K_{3,3}))=9$.
	d) If $A_1 \neq \emptyset$ and $A_2 \neq \emptyset$ then $Q_{M^*(K_5)}\backslash A_1/A_2 \cong M^*(K_{3,3})$, a contradiction as $E(M^*(K_{3,3}))=9$ and $E(Q_{M^*(K_5)}\backslash A_1/A_2) \leq 8$.
	Thus, $M^*(K_{3,3})$ is not a minor of $Q_{M^*(K_5)}$ and hence by Theorem \ref{gt}, we say that $Q_{M^*(K_5)}$ is graphic.
\end{proof}

\begin{lem}\label{qk33}
	Let $a\in E(N)$, where $N$ is a binary matroid, such that $a$ is not a loop or a coloop, then $Q_{M^*(K_{3,3})}$ not containing $F_7$ or $F_7^*$ is graphic. 
\end{lem}
\begin{proof}
	Suppose $a\in E(N)$, where $N$ is a binary matroid, such that $a$ is not a coloop or a loop such that $N\backslash a \cong M^*(K_{3,3})$ then $E(N\backslash a)=9$ and $r(N\backslash a)=4$ then $r(N)=4$, $E(N)=10$. 
	We have $Q_{M^*(K_{3,3})}=N/a$, then $r(Q_{M^*(K_{3,3})})=3$, $E(Q_{M^*(K_{3,3})})=9$. Suppose $Q_{M^*(K_{3,3})}$ is not graphic. Then by Theorem \ref{gt}, $Q_{M^*(K_{3,3})}$ has a minor $M^*(K_5)$ or $M^*(K_{3,3})$.\\
	Case(i) If $Q_{M^*(K_{3,3})}$ contains $M^*(K_5)$, then $Q_{M^*(K_{3,3})} \backslash A_1/A_2 \cong M^*(K_5)$ for some subsets $A_1$ or $A_2$ of $E(Q_{M^*(K_{3,3})})$, which is a contradiction, as $r(Q_{M^*(K_{3,3})} \backslash A_1/A_2)\leq 3$ however $r(M^*(K_5))=6$ \\
	Case(ii) If $Q_{M^*(K_{3,3})}$ has a minor $M^*(K_{3,3})$, then $Q_{M^*(K_{3,3})} \backslash A_1/A_2 \cong M^*(K_{3,3})$ for some subsets $A_1$ or $A_2$ of $E(Q_{M^*(K_{3,3})})$, which is a contradiction, as $r(Q_{M^*(K_{3,3})} \backslash A_1/A_2)\leq 3$ however $r(M^*(K_{3,3}))=4$. \\
	Thus from the case(i) and case(ii) and by Theorem \ref{gt}, we say that $Q_{M^*(K_{3,3})}$ is graphic. 
\end{proof}
  
\begin{lem}\label{qM*K5}
$M^*(K_5)$ is the non-graphic quotient $Q_{M^*(K_5)}$, not containing $F_7$ and $F_7^*$.	
\end{lem}
\begin{proof}
	Let $a\in E(N)$, where $N$ is a binary matroid such that $N\backslash a \cong M^*(K_5)$.
	(i) If $a$ is a coloop or a loop. Then $N/a \cong N\backslash a \cong M^*(K_5)$. \\
	(ii) If $a$ is not a coloop or loop then by Lemma \ref{qk5}, $N/a$ is graphic. \\
Thus from above $M^*(K_5)$ is the only non-graphic elementary quotient of $M^*(K_5)$  not containing $F_7$ and $F_7^*$.
 \end{proof}
\begin{lem}\label{qM*K33}
$M^*(K_{3,3})$ is the non-graphic quotient $Q_{M^*(K_{3,3})}$, not containing $F_7$ and $F_7^*$.	
\end{lem}
\begin{proof}
	Let $a\in E(N)$, where $N$ is a binary matroid such that $N\backslash a \cong M^*(K_{3,3})$.
	(i) If $a$ is a coloop or a loop, then $N\backslash a \cong N/a \cong  M^*(K_{3,3})$. \\
	(ii) If $a$ is not a coloop or loop then by Lemma \ref{qk33}, $N/a$ is graphic. \\
	Thus from above we say that, $M^*(K_{3,3})$ is the only non-graphic elementary quotient of $M^*(K_{3,3})$ not containing $F_7$ and $F_7^*$.
	
\end{proof}

\section{Main Theorems}
In the previous section, we mentioned the graphic and non-graphic quotients for every $F \in \mathcal{F}$. Now, the main theorems are proved in this section.  
\begin{thm}
A cographic binary matroid $M \notin \mathcal{C}_k$, $k \geq 2$, then $M$ contains a minor $P$ such that one of the below is true.\\
i) $P$ is an extension of a minimal minor $E$  by single element, such that $E \notin \mathcal{C}_{k-1}$. \\
ii)  $P= M(Q_i)$.\\
iii) $P$ is a coextension of $M(Q_i)$ by $n$ elements, where $n \leq k$, the Figure \ref{quotient} shows the graph $Q_i$, for $i=1,2, \cdots 9$.    

\end{thm} 
\begin{proof}
	Let a binary cographic matroid be $M$ such that $M \notin \mathcal{C}_k$, $k \geq 2$, that is for $S \subseteq E(M)$, with $|S|=k$, $M_S$ is non-graphic matroid. From Lemma \ref{mainlemma}, $M$ has a minor $P$ with $S \subseteq E(P)$, such that  $P_S \cong F$ or $P_S/S' \cong F$, for some $S'\subseteq S$ and $F \in \mathcal{F}$ or $P$ is an extension of some a minimal minor $E$ by single element, such that $E \notin \mathcal{C}_{k-1}$. If  $P_S \cong F$ or $P_S/S' \cong F$, for some $S'\subseteq S$ then by Lemma \ref{rel}, either $P \cong Q_F$ or $P$ is extension of $Q_F$ by $n$ elements, where $n \leq k$ and $Q_F=N/a$ is a quotient of $F \in \mathcal{F}$.\\
	Case (i) If the quotient is graphic, then 
	a) If $F=F_7^*$, then by Lemma \ref{qF7*},  $Q_{F_7^*} \cong M(Q_1)$ or $Q_{F_7^*} \cong M(Q_2)$. b) If $F=F_7$, then by Lemma \ref{qF7}, $Q_{F_7} \cong M(Q_3)$. If $F=M^*(K_{3,3})$, then by Lemma \ref{gqm*k33}, $Q_{M^*(K_{3,3})} \cong M(Q_4)$ or $Q_{M^*(K_{3,3})} \cong M(Q_5)$. If $F=M^*(K_5)$ then by Lemma \ref{gqm*k5}, $Q_{M^*(K_5)} \cong M(Q_i)$. Figure \ref{quotient} shows the graph $Q_i$, for $i=1,2, \cdots 9$.\\ 
	Case (ii) If the quotient is not graphic, then by Lemma \ref{qM*K5}, a non-graphic quotient $Q_{M^*(K_5)}=M^*(K_5)$ and by Lemma \ref{qM*K33}, a non-graphic quotient $Q_{M^*(K_{3,3})} = M^*(K_{3,3})$. From Figure \ref{M*G1}, $M^*(Q_1)$ is a minor of the matroid $M(K_5)$, thus $M(Q_1)$ is a minor of the matroid $M^*(K_5)$ and From the Figure \ref{M*G2}, $M^*(Q_2)$ is a minor of the matroid $M(K_{3,3})$, thus $M(Q_2)$ is a minor of the matroid $M^*(K_{3,3})$. Hence we discard non-graphic quotients. 
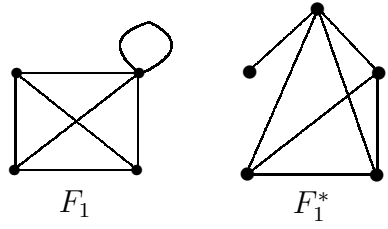
\begin{figure}[h!]
	\unitlength 1mm 
	\linethickness{0.4pt}
	\ifx\plotpoint\undefined\newsavebox{\plotpoint}\fi 
	\begin{picture}(54.902,28.184)(0,0)
		\put(6.184,6.108){\circle*{1.514}}
		\put(22.266,6.079){\circle*{1.514}}
		\put(6.475,18.937){\circle*{1.514}}
		\put(22.558,18.908){\circle*{1.514}}
		\put(6.348,18.942){\line(1,0){16.2}}
		\put(6.281,18.939){\line(0,-1){12.712}}
		\put(22.369,6.001){\line(0,1){13.163}}
		\put(6.373,6.108){\line(1,0){16.271}}
		\multiput(6.278,18.974)(.0418534031,-.0336806283){382}{\line(1,0){.0418534031}}
		\multiput(5.995,6.108)(.0433376963,.0336806283){382}{\line(1,0){.0433376963}}
		\qbezier(22.302,19.168)(17.015,23.556)(23.877,25.693)
		\qbezier(23.877,25.693)(30.402,22.431)(22.527,18.718)
		\put(14.142,1.591){\makebox(0,0)[cc]{$F_1$}}
		\put(36.811,5.499){\circle*{1.599}}
		\put(53.795,5.468){\circle*{1.599}}
		\put(37.12,19.046){\circle*{1.599}}
		\put(54.102,19.016){\circle*{1.599}}
		\put(53.903,5.386){\line(0,1){13.9}}
		\put(37.011,5.499){\line(1,0){17.182}}
		\put(45.995,27.384){\circle*{1.599}}
		\multiput(46.044,27.574)(.03366337,-.036074757){235}{\line(0,-1){.036074757}}
		\multiput(36.895,19.152)(.0365904,.0337392){250}{\line(1,0){.0365904}}
		\multiput(36.658,5.608)(.0337021277,.0779361702){282}{\line(0,1){.0779361702}}
		\multiput(46.036,27.557)(.033642478,-.097247788){226}{\line(0,-1){.097247788}}
		\multiput(36.652,5.46)(.042847791,.0336895522){402}{\line(1,0){.042847791}}
		\put(45.25,1.25){\makebox(0,0)[cc]{$F_1^*$}}
	\end{picture}
	\caption{The Graphs $F_1 \cong Q_1$ and $F_1^* \cong Q_1^*$.}
	\label{M*G1}
	
\end{figure}

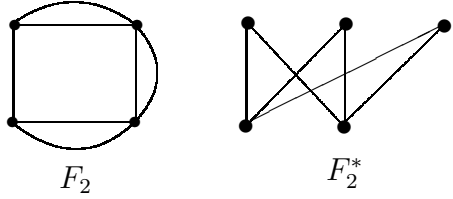
\begin{figure}[h!]
	\unitlength 1mm 
	\linethickness{0.4pt}
	\ifx\plotpoint\undefined\newsavebox{\plotpoint}\fi 
	\begin{picture}(65.824,29.069)(0,0)
		\put(38.873,9.634){\circle*{1.615}}
		\put(51.788,9.4){\circle*{1.615}}
		\put(39.184,23.318){\circle*{1.615}}
		\put(52.1,23.085){\circle*{1.615}}
		\put(38.977,23.32){\line(0,-1){13.56}}
		\put(51.899,9.319){\line(0,1){14.04}}
		\put(65.016,22.883){\circle*{1.615}}
		\multiput(38.873,9.602)(.0337197943,.0347583548){389}{\line(0,1){.0347583548}}
		\multiput(38.873,23.527)(.0337241379,-.036403183){377}{\line(0,-1){.036403183}}
		\multiput(51.587,9.803)(.0342416452,.0337249357){389}{\line(1,0){.0342416452}}
		\put(38.468,9.803){\line(2,1){26.639}}
		\put(8.207,10.149){\circle*{1.514}}
		\put(24.289,10.12){\circle*{1.514}}
		\put(8.499,22.978){\circle*{1.514}}
		\put(24.581,22.949){\circle*{1.514}}
		\put(8.371,22.983){\line(1,0){16.2}}
		\put(8.304,22.979){\line(0,-1){12.713}}
		\put(24.392,10.042){\line(0,1){13.162}}
		\qbezier(8.207,23.203)(16.626,29.069)(24.668,22.825)
		\put(8.396,10.149){\line(1,0){16.271}}
		\qbezier(8.018,10.149)(16.911,2.959)(24.289,10.149)
		\qbezier(24.55,22.775)(30.062,16.475)(24.325,10.175)
		\put(16.5,2.527){\makebox(0,0)[cc]{$F_2$}}
		\put(51.879,3.419){\makebox(0,0)[cc]{$F_2^*$}}
	\end{picture}
	
\caption{The Graphs $F_2 \cong Q_2$ and $F_2^* \cong Q_2^*$}
	\label{M*G2}
\end{figure}
	
	 Thus from above, either the minor $P=Q_F$ or a coextension of $Q_F$ not more than $k$ elements for $F \in \mathcal{F}$. Hence the result.   
\end{proof}
\noindent We now obtain excluded minors for the class $\mathcal{C}_3$. 
\begin{thm}
Let a cographic binary matroid be $M$, then $M \in \mathcal{C}_3$ if and only if $M$ does not have a minor $M(F_i)$, where the Figure \ref{g2g} shows the graph $F_i$, for $i=1,2, \cdots 7$.

\end{thm}
\begin{proof}
Suppose a cographic matroid $M$ contains minor $M(F_i)$, for $i=1,2, \cdots 7$, then $M \notin \mathcal{C}_3$, the proof is straight forward.\\
Conversely, if $M$ does not contain a minor $M(F_i)$ for $i=1,2, \cdots 7$, then we will prove that $M\in \mathcal{C}_3$. Suppose not, then for some $S\subseteq E(M)$, with $|S|=3$, $M_S$ is not a graphic matroid, then, $M_S$ contains minor $F$, for some $F \in \mathcal{F}$, by Theorem \ref{gt}. Then $M$ contains a minor $P$ containing $S$, By Lemma \ref{mainlemma}, such that  $P_S \cong F$ or $P_S/S' \cong F$, for some $S'\subseteq S$ or $P$ is an extension of circuit matroid of the graph $G_1$ or $G_2$ by single element and the graphs $G_1$, $G_2$ are given in Figure \ref{Fig_gtg_2elt}.
It is observed that an extension of $M(G_1)$ by a single element, either isomorphic to $M(F_4)$ or $M(F_7)$ or contains minor $M(F_1)$ or $M(F_2)$. Also, $M(G_2)$ contains minor $M(F_2)$. Thus $P$ cannot be an extension of $M(G_1)$ or $M(G_2)$ by a single element. \\
Hence, $P_S \cong F$ or $P_S/S' \cong F$. Then by Lemma \ref{rel}, either  $P$ is an extension of $Q_F$ by $n$ elements, where $n \leq 3$ or $P \cong Q_F$. \\
Case (i) If the quotient is graphic. In \cite{gm}, Mundhe et al. obtained forbidden minors from graphic quotients of every $F \in \mathcal{F}$, as given in Theorem \ref{gtg3thm}.
Case (ii) If the quotient is not graphic. Let $F=M^*(K_5)$, then by Lemma \ref{qM*K5}, $Q_{M^*(K_5)} \cong M^*(K_5)$ but from Figure \ref{M*G1}, $M^*(F_1)$ is a minor of $M(K_5)$ and hence $M(F_1)$ is a minor of $M^*(K_5)$. Hence we discard $M^*(K_5)$. Let $F=M^*(K_{3,3})$, then $Q_{M^*(K_{3,3})} \cong M^*(K_{3,3})$, by Lemma \ref{qM*K33}, but from Figure \ref{M*G2},  $M^*(F_2)$ is a minor of $M(K_{3,3})$ and hence $M(F_2)$ is a minor of $M^*(K_{3,3})$. Hence, we discard $M^*(K_{3,3})$. \\
Thus by the case (i) and case (ii), the excluded minor for the class $\mathcal{C}_3$ is the matroid $M(F_i)$, the graph $F_i$ is shown in Figure \ref{g2g}, for $i=1,2, \cdots 7$. 

\end{proof}
\bibliographystyle{amsplain}

\end{document}